\documentclass[a4paper]{amsart}
\usepackage{graphicx}
\usepackage{amsmath}
\usepackage{amssymb}
\usepackage{oldgerm}
\usepackage[all]{xy}

\newcommand{\Hom}{\operatorname{Hom}\nolimits}
\newcommand{\End}{\operatorname{End}\nolimits}
\newcommand{\uHom}{\operatorname{\underline{Hom}}\nolimits}
\newcommand{\Aut}{\operatorname{Aut}\nolimits}
\renewcommand{\Im}{\operatorname{Im}\nolimits}
\renewcommand{\mod}{\operatorname{mod}\nolimits}
\newcommand{\Ker}{\operatorname{Ker}\nolimits}

\newcommand{\Ext}{\operatorname{Ext}\nolimits}

\newcommand{\HH}{\operatorname{HH}\nolimits}

\newcommand{\rad}{\operatorname{rad}\nolimits}
\newcommand{\soc}{\operatorname{soc}\nolimits}
\newcommand{\Irr}{\operatorname{Irr}\nolimits}

\newcommand{\La}{\Lambda}

\newcommand{\V}{\operatorname{V}\nolimits}

\newcommand{\Tr}{\operatorname{Tr}\nolimits}

\newcommand{\Char}{\operatorname{char}\nolimits}

\newtheorem{theorem}{Theorem}[section]

\newtheorem{lemma}[theorem]{Lemma}

\theoremstyle{definition}

\theoremstyle{definition}

\theoremstyle{definition}
\newtheorem*{setup}{Setup}

\theoremstyle{definition}

\theoremstyle{remark}

\theoremstyle{remark}

\theoremstyle{definition}

\begin{document}
\title[The stable AR-quiver of a quantum complete intersection]{The stable Auslander-Reiten quiver of a quantum complete intersection}
\author{Petter Andreas Bergh \& Karin Erdmann}
\address{Petter Andreas Bergh \\ Institutt for matematiske fag \\
NTNU \\ N-7491 Trondheim \\ Norway} \email{bergh@math.ntnu.no}
\address{Karin Erdmann \\ Mathematical Institute \\ 24-29 St.\ Giles \\ Oxford OX1 3LB \\ United Kingdom}
\email{erdmann@maths.ox.ac.uk}

\thanks{The first author was supported by NFR Storforsk grant no.\
167130}



\begin{abstract}
We completely describe the tree classes of the components of the stable Auslander-Reiten quiver of a quantum complete intersection. In particular, we show that the tree class is always $A_{\infty}$ whenever the algebra is of wild representation type. Moreover, in the tame case, there is one component of tree class $\tilde{A}_{12}$, whereas all the others are of tree class $A_{\infty}$.
\end{abstract}

\maketitle

\section{Introduction}

The Auslander-Reiten quiver of a finite dimensional algebra encapsulates the information on the indecomposable modules and the irreducible maps. Describing these quivers for various algebras is therefore one of the classical problems in representation theory.

In \cite{Webb}, Webb studied the stable Auslander-Reiten quiver of group algebras of finite groups over algebraically closed fields. He showed that the tree class of any stable component is either $A_n$, a Euclidean diagram, or one of the infinite trees $A_{\infty}$, $B_{\infty}$, $C_{\infty}$, $D_{\infty}$ and $A_{\infty}^{\infty}$. A natural question to ask was then: do any of these tree classes never occur? The second author answered this completely in \cite{Erdmann2}, by showing that any stable component of a wild block of a group algebra is of tree class $A_{\infty}$. 

Not much is known in general for arbitrary selfinjective algebras over algebraically closed fields. However, when restricting to modules of finite complexity, some results have been obtained. In \cite{GreenZacharia}, Green and Zacharia showed that for a selfinjective algebra, 
a regular component containing a module of complexity $1$ is either a stable tube or has tree class $A_{\infty}$. For stable components containing modules of higher (finite) complexity, the picture is more complicated. Kerner and Zacharia showed in \cite{KernerZacharia} that the tree class of a non-$\tau$-periodic such stable component is either a Euclidean diagram, or one of the infinite trees  $A_{\infty}$, $D_{\infty}$ and $A_{\infty}^{\infty}$. Moreover, if the component is regular, then the tree class is one of the infinite trees. A similar result was also shown by the second author and coauthors in \cite{EHSST}, where they studied selfinjective algebras having finitely generated cohomology.

In this paper, we completely describe the tree classes of the stable Auslander-Reiten components of the quantum complete intersections. In particular, we show that if such an algebra is wild, which they ``almost always" are, then every stable component is of tree class $A_{\infty}$. Moreover, in the tame case, there is one component of tree class $\tilde{A}_{12}$, whereas all the others are of tree class $A_{\infty}$.

\section{Preliminaries}

Throughout this section, let $\La$ be a finite dimensional algebra over some algebraically closed field. Denote by $\mod \La$ the category of finitely generated left $\La$-modules; all modules encountered are assumed to belong here.

The Auslander-Reiten quiver $\Gamma ( \La )$ of $\La$ (from now on called ``AR-quiver") is a directed graph, whose vertices are the isomorphism classes $[M]$ of the indecomposable $\La$-modules. The edges correspond to the irreducible morphisms; there is an arrow $[M] \to [N]$ in $\Gamma ( \La )$ precisely when there exists an irreducible morphism $M \xrightarrow{f} N$. For every such edge one associates a pair of numbers 
$$[M] \xrightarrow{(a_{MN},b_{MN})} [N]$$
corresponding to the $\End_{\La}(M)$-$\End_{\La}(N)$-bimodule 
$$\Irr_{\La}(M,N) \stackrel{\text{def}}{=} \rad_{\La}(M,N)/ \rad_{\La}^2(M,N)$$ 
of irreducible morphisms. Namely, the number $a_{MN}$ is the length of $\Irr_{\La}(M,N)$ as a right $\End_{\La}(N)$-module, whereas $b_{MN}$ is the length of $\Irr_{\La}(M,N)$ as a left $\End_{\La}(M)$-module. It is customary to omit the pair in the case when both numbers are $1$.

Now suppose the algebra is selfinjective. The stable AR-quiver of $\La$, denoted $\Gamma_s ( \La )$, is the subquiver of $\Gamma ( \La )$ obtained by deleting all the projective modules. By the Riedtmann Structure Theorem (cf.\ \cite{Riedtmann}), given any connected component $\Theta$ of $\Gamma ( \La )$, there exists a directed tree $T$ and an admissible group of automorphisms $G \subseteq \Aut ( \mathbb{Z} T)$, such that $\Theta$ is isomorphic to $\mathbb{Z} T/G$. The admissible group $G$ is uniquely defined up to conjugation in $\Aut ( \mathbb{Z} T)$. Moreover, the graph associated to $T$, that is, the graph obtained from $T$ by replacing each arrow by an edge, is determined uniquely by $\Theta$ up to canonical isomorphism. The isomorphism type of this graph is the \emph{tree class} of the component $\Theta$. 

It is well known that if a connected component $\Theta$ of $\Gamma_s ( \La )$ contains a $\tau$-periodic module (i.e.\ a module $M$ such that $\tau^nM \simeq M$ for some $n>0$), then $\Theta$ itself is $\tau$-periodic, that is, all modules in this component are $\tau$-periodic. In this case, if $\La$ is of infinite representation type, then the component $\Theta$ is a tube by \cite[Main Theorem]{HappelPreiserRingel}, and its tree class is therefore $A_{\infty}$. Thus, in order to determine tree classes of stable components for selfinjective algebras of infinite representation type, it suffices to consider only those components which are not $\tau$-periodic. When doing this, we shall be working with additive functions on components of the stable AR-quiver of $\La$, that is, functions which are additive on AR-sequences. For a given $\La$-module $W$, define a function $d_W \colon \mod \La \to \mathbb{N} \cup \{ 0 \}$ by 
$$d_W(M) \stackrel{\text{def}}{=} \dim \uHom_{\La} (W,M).$$
A priori, this function is not additive on a given stable component. However, the following result from \cite{ErdmannSkowronski} shows that additivity holds whenever no indecomposable summand of $W$ belongs to the component or its shift.

\begin{lemma}\cite[Lemma 3.2]{ErdmannSkowronski}\label{additive}
Let $\La$ be a finite dimensional selfinjective algebra over an algebraically closed field, and $W$ be a $\La$-module. Furthermore, let $\Theta$ be a component of $\Gamma_s ( \La )$, and suppose that no indecomposable summand of $W$ belongs to $\Theta$ or $\Omega_{\La}^1( \Theta )$. Then $d_W$ is an additive function on $\Theta$.
\end{lemma}

To identify tree classes of stable components, we have to work with additive functions which are constant on $\tau$-orbits. When $W \simeq \tau W$, then $d_W$ satisfies the latter requirement automatically, since then
$$\uHom_{\La} (W, \tau M) \simeq \uHom_{\La} ( \tau W, \tau M) \simeq \uHom_{\La} (W, M).$$
In this situation, the following result guarantees that $d_W$ is additive and constant on the tree of a stable component of tree class $A_{\infty}^{\infty}$ or $D_{\infty}$.

\begin{lemma}\label{tauconstant}
Let $\La$ be a finite dimensional selfinjective algebra over an algebraically closed field, and $W$ be a $\La$-module such that $\tau W \simeq W$. Furthermore, let $\Theta$ be a component of $\Gamma_s ( \La )$.
\begin{enumerate}
\item If $\Theta$ is of tree class $A_{\infty}^{\infty}$, then $d_W$ is additive and constant on its tree.
\item If $\Theta$ is of tree class $D_{\infty}$, then $d_W$ is additive and constant on all modules of the tree of $\Theta$ having at least two predecessors.
\end{enumerate}
\end{lemma}

\begin{proof}
Every indecomposable summand $W_0$ of $W$ must be $\tau$-periodic, i.e.\ $\tau^nW_0 \simeq W_0$ for some $n \ge 1$. Suppose $\Theta$ (or $\Omega_{\La}^1( \Theta )$) contains a $\tau$-periodic module. Then $\Theta$ (or $\Omega_{\La}^1( \Theta )$) itself is $\tau$-periodic, that is, every module in $\Theta$ (or $\Omega_{\La}^1( \Theta )$) is $\tau$-periodic. Moreover, by \cite[Main Theorem]{HappelPreiserRingel}, the tree class of $\Theta$ is either a Dynkin diagram or $A_{\infty}$. Therefore, when $\Theta$ is of tree class $A_{\infty}^{\infty}$ or $D_{\infty}$, then it cannot contain a $\tau$-periodic module, and neither can $\Omega_{\La}^1( \Theta )$. In particular, no indecomposable summand of $W$ belongs to $\Theta$ or $\Omega_{\La}^1( \Theta )$. It follows from Lemma \ref{additive} that $d_W$ is additive on $\Theta$, and we know that this function is constant on $\tau$-orbits. Thus $d_W$ is additive on the tree of $\Theta$. The result now follows from standard arguments applied to additive functions on $A_{\infty}^{\infty}$ and $D_{\infty}$, cf.\ \cite[Proof of Lemma 2.30.5]{Benson1} or \cite[Proof of Lemma 3]{HappelPreiserRingel}.
\end{proof}

We end this section with two results which are vital for the proof of the main theorem. Note the resemblance to \cite[Proposition 1.5]{Erdmann2}.

\begin{lemma}\label{Extisomorphism}
Let $\La$ be a finite dimensional selfinjective algebra over an algebraically closed field, and $W$ be a $\La$-module such that $\tau W \simeq W$. Furthermore, let
$$0 \to L \xrightarrow{f} M \xrightarrow{g} N \to 0$$
be a non-split exact sequence in $\mod \La$ satisfying the following:
\begin{enumerate}
\item The modules $M$ and $N$ are indecomposable.
\item The map $g$ is irreducible.
\item The stable AR-component containing $N$ is of tree class $A_{\infty}^{\infty}$ or $D_{\infty}$, and $N$ has at least two predecessors.
\item There exists no monomorphism $L \to W_0$ for any indecomposable summand $W_0$ of $W$.
\end{enumerate}
Then for any indecomposable summand $W_0$ of $W$, the induced map 
$$\Ext_{\La}^1(M,W_0) \xrightarrow{f^*} \Ext_{\La}^1(L,W_0)$$ is zero, and $\Ext_{\La}^1(M,W_0)$ is isomorphic to $\Ext_{\La}^1(N,W_0)$.
\end{lemma}

\begin{proof}
Since $W \simeq \tau W$, the AR-formula gives
$$\dim \uHom_{\La} (W,X) = \dim \Ext_{\La}^1( \tau^{-1}X,W) = \dim \Ext_{\La}^1( X, \tau W) = \dim \Ext_{\La}^1( X, W)$$
for any $\La$-module $X$. Therefore, from Lemma \ref{tauconstant} and the assumptions (1), (2) and (3), we conclude that the dimensions of $\Ext_{\La}^1(M,W)$ and $\Ext_{\La}^1(N,W)$ are equal. 

Now let $W_0$ be any indecomposable summand of $W$. Applying $\Hom_{\La}(-,W_0)$ to the exact sequence gives an exact sequence
$$\Hom_{\La}(M,W_0) \xrightarrow{f^*} \Hom_{\La}(L,W_0) \to \Ext_{\La}^1(N,W_0) \xrightarrow{g^*} \Ext_{\La}^1(M,W_0)$$
of vector spaces. By assumption (4) and \cite[Proposition 1.1]{Erdmann2}, the leftmost map is surjective, hence the rightmost map is injective. This shows that $\dim \Ext_{\La}^1(N,W_0) \le \dim \Ext_{\La}^1(M,W_0)$. From above we conclude that the dimensions must be equal, and so the rightmost map is an isomorphism. Finally, the exact sequence
$$\Ext_{\La}^1(N,W_0) \xrightarrow{g^*} \Ext_{\La}^1(M,W_0) \xrightarrow{f^*} \Ext_{\La}^1(L,W_0)$$
implies that the map $\Ext_{\La}^1(M,W_0) \xrightarrow{f^*} \Ext_{\La}^1(L,W_0)$ is zero.
\end{proof}

\begin{lemma}\label{Extvanishing}
With the same setup and hypotheses as in \emph{Lemma \ref{Extisomorphism}}, let $W_0$ be an indecomposable summand of $W$. Then an element $\eta \in \Ext_{\La}^1(L,W_0)$ is zero whenever it can be represented by a map $L \xrightarrow{\eta} \Omega_{\La}^{-1}(W_0)$ which is not a monomorphism. In particular, if there is no monomorphism $L \to \Omega_{\La}^{-1}(W_0)$, then $\Ext_{\La}^1(L,W_0)=0$.
\end{lemma}

\begin{proof}
Consider the diagram
$$\xymatrix{
0 \ar[r] & L \ar[r]^f \ar[d]^{\eta} & M \ar[r]^g & N \ar[r] & 0 \\
& \Omega_{\La}^{-1}(W_0) }$$
in $\mod \La$. By \cite[Proposition 1.1]{Erdmann2}, there is either a map $\Omega_{\La}^{-1}(W_0) \xrightarrow{\phi} M$ with $f = \phi \eta$, or a map $M \xrightarrow{\psi} \Omega_{\La}^{-1}(W_0)$ with $\eta = \psi f$. The first situation cannot occur, since $\eta$ is not a monomorphism. Hence there must exist a $\psi$, and this shows that $\eta$ is in the image of the map $\Ext_{\La}^1(M,W_0) \xrightarrow{f^*} \Ext_{\La}^1(L,W_0)$. However, by the previous lemma, this map is zero.
\end{proof}

\section{Quantum complete intersections}

The quantum complete intersections are a class of algebras originating from work by Manin and Avramov, Gasharov and Peeva (cf.\ \cite{Manin} and \cite{AvramovGasharovPeeva}). These are analogues of truncated polynomial rings, in that they are quantum polynomial rings modulo quantum regular sequences.

Throughout this section, we fix an algebraically closed field $k$. Let $c \ge 1$ be an integer, and let ${\bf{q}} = (q_{ij})$ be a $c \times c$ commutation matrix with entries in $k$. That is, the diagonal entries $q_{ii}$ are all $1$, and $q_{ij}q_{ji}=1$ for $i \neq j$. Furthermore, let ${\bf{a}}_c = (a_1, \dots, a_c)$ be an ordered sequence of $c$ integers with $a_i \ge 2$. The \emph{quantum complete intersection} $A_{\bf{q}}^{{\bf{a}}_c}$ determined by these data is the algebra
$$A_{\bf{q}}^{{\bf{a}}_c} \stackrel{\text{def}}{=} k \langle x_1, \dots, x_c \rangle / (x_i^{a_i}, x_ix_j-q_{ij}x_jx_i),$$
which is selfinjective and finite dimensional of dimension $\prod a_i$. 

The homological behavior of different quantum complete intersections
can vary enormously, depending on whether or not the defining
commutators are roots of unity. For example, in the case with two
generators and where the commutator is not a root of unity, it was
shown in \cite{BerghErdmann1} and \cite{BuchweitzGreenMadsenSolberg}
that all the higher Hochschild cohomology groups vanish. On the other
hand, when all the commutators are roots of unity, then it was proved
in \cite{BerghOppermann} that the even Hochschild cohomology ring
$\HH^{2*}( A_{\bf{q}}^{{\bf{a}}_c} )$ is Noetherian, and
$\Ext_{A_{\bf{q}}^{{\bf{a}}_c}}^*(M,N)$ is a finitely generated
$\HH^{2*}( A_{\bf{q}}^{{\bf{a}}_c} )$-module for all
$A_{\bf{q}}^{{\bf{a}}_c}$-modules $M$ and $N$. Thus, in this case, one
has a very rich theory of support varieties with respect to $\HH^{2*}
(A_{\bf{q}}^{{\bf{a}}_c})$, as developed in \cite{EHSST} and
\cite{SnashallSolberg}. For example, the support varieties detect
projective and periodic modules.

Not surprisingly, the representation type of the quantum
complete intersections also depend on the defining parameters. Suppose
all the commutators $q_{ij}$ are roots of unity.
Then by \cite[Theorem 5.3]{BerghOppermann} and \cite[Theorem
5.5]{BerghOppermann}, the 
growth rate of the $\Ext$-algebra of the $A_{\bf{q}}^{{\bf{a}}_c}$-module $k$ is
$c$. Moreover, this cohomology algebra is module-finite over its own
center, and the latter is a Noetherian ring. Hence, if $c \ge 3$, then it
follows from \cite[Theorem 4.1]{BerghSolberg} that
$A_{\bf{q}}^{{\bf{a}}_c}$ is of wild 
representation type. If $c=2$, then our algebra
$A_{\bf{q}}^{{\bf{a}}_c}$ is of the form 
$$k \langle x,y \rangle / (x^a, xy-qyx, y^b).$$
In this situation, if $b \ge 3$, then the algebra
$$k \langle x,y \rangle / (x^2, xy-qyx, y^3, y^2x)$$
is a factor algebra of $A_{\bf{q}}^{{\bf{a}}_c}$. This factor algebra
is wild by \cite[3.4]{Ringel}, hence so is
$A_{\bf{q}}^{{\bf{a}}_c}$. By symmetry, the same 
conclusion holds when $a \ge 3$. Finally, if $a$, $b$ and $c$ all equal $2$, then
$A_{\bf{q}}^{{\bf{a}}_c}$ is a factor algebra of the algebra
$$k \langle x,y \rangle / (x^2, y^2),$$
and so $A_{\bf{q}}^{{\bf{a}}_c}$ is tame by \cite[1.3]{Ringel}.

The main result of this paper completely describes the components of
the stable AR-quiver of a quantum complete intersection
$A_{\bf{q}}^{{\bf{a}}_c}$ in which all the defining exponents are
equal, i.e.\ of the form
$$k \langle x_1, \dots, x_c \rangle / (x_i^a, x_ix_j-q_{ij}x_jx_i).$$
However, we shall be working with the quantum complete
intersections having 
a well defined notion of rank varieties for modules, namely the
homogeneous ones. 
We therefore fix such an algebra. 
\begin{setup}
\begin{enumerate}
\item Fix integers $c,a \ge 2$.
\item Define $b \stackrel{\text{def}}{=} \left \{ \begin{array}{ll}
a/ \gcd ( a, \Char k ) & \text{if } \Char k >0 \\
a & \text{if } \Char k =0.
\end{array} \right.$
\item Let $q \in k$ be a primitive $b$th root of unity.
\item Define $A \stackrel{\text{def}}{=} k \langle x_1, \dots, x_c \rangle / ( \{ x_i^a \}, \{ x_ix_j-qx_jx_i \}_{i<j} ).$
\end{enumerate}
\end{setup}

The rank variety of an $A$-module is given in terms of certain commutative subalgebras. Namely, given any $c$-tuple $\lambda = ( \lambda_1, \dots, \lambda_c ) \in k^c$, denote the element $\lambda_1 x_1 + \cdots + \lambda_c x_c \in A$ by $u_{\lambda}$, and let $k[ u_{\lambda} ]$ be the subalgebra of $A$ generated by this element. Then $u_{\lambda}^a =0$ by \cite[Lemma 2.3]{BensonErdmannHolloway}, and the \emph{rank variety} of an $A$-module $M$, denoted $\V^r_{A}(M)$, is defined as
$$\V^r_{A}(M) \stackrel{\text{def}}{=} \{ 0 \} \cup \{ 0 \neq \lambda \in k^c \mid M \text{ is not a projective $k[ u_{\lambda} ]$-module} \}.$$
When $\lambda$ is nonzero, then since $u_{\lambda}^a =0$, the subalgebra $k[ u_{\lambda} ]$ is isomorphic to the truncated polynomial ring $k[x]/(x^a)$. Therefore, the requirement that an $A$-module $M$ is not $k[ u_{\lambda} ]$-projective is equivalent to the requirement that the $k$-linear map $M \xrightarrow{\cdot u_{\lambda}} M$ satisfies
$$\dim \Im ( \cdot u_{\lambda} ) < ((n-1)/n) \dim M.$$
This explains the choice of terminology, and shows that rank varieties are homogeneous affine subsets of $k^c$.

As a first step in determining the tree classes of stable AR-components of quantum complete intersections, we prove the following lemma. It shows that when the algebra is wild, i.e. when either $a \ge 3$ or $c \ge 3$, then the tree class of a component cannot be a Euclidean diagram. 

\begin{lemma}\label{noteuclidean}
If either $a \ge 3$ or $c \ge 3$, then the tree class of any stable AR-component of $A$ cannot be one of the Euclidean diagrams $\tilde{A}_n$, $\tilde{D}_n$, $\tilde{E}_6$, $\tilde{E}_7$ and $\tilde{E}_8$.
\end{lemma} 

\begin{proof}
Let $\Theta$ be a component of the stable AR-quiver, and suppose its tree class is one of the given Euclidean diagrams. Then it follows from \cite[Main Theorem]{KernerZacharia} that $\Theta$ is not regular, i.e.\ it must be the non-regular component (there is only one such component, since $A$ is local). Now if either $a \ge 3$ or $c \ge 3$, then the $A$-module $\rad A / \soc A$ is indecomposable, hence the middle term of the AR-sequence starting in $\rad A$ is $A \oplus \rad A / \soc A$. Thus $\rad A$ sits at the end of the component. Now since $\Theta$ has an end at all, its tree class must be either $\tilde{E}_6$, $\tilde{E}_7$ or $\tilde{E}_8$. By comparing dimensions modulo $\dim A$, we see that the $A$-module $k$ belongs to a different $\tau$-orbit than $\rad A$, and also sits at the end. This cannot happen if the tree class is $\tilde{E}_8$. Similar dimension arguments also rule out the other two possibilities $\tilde{E}_6$ and $\tilde{E}_7$. Consequently, the tree class cannot be one of the given Euclidean diagrams.
\end{proof}

By combining this lemma with \cite[Main Theorem]{KernerZacharia}, we see that when $A$ is wild, then the only possible tree classes are $A_{\infty}$, $A_{\infty}^{\infty}$ and $D_{\infty}$. So in order to show that only the first of these occur, we rule out the other two possibilities in Theorem \ref{exponent} and Theorem \ref{codimension}. As a first step, we present the following result, which deals with irreducible maps between $A$-modules whose stable AR-component has tree class $A_{\infty}^{\infty}$ or $D_{\infty}$. It
shows that the kernel of such a map is periodic, and that all its shifts have simple tops and have dimensions divisible by $a$.

\begin{lemma}\label{periodic}
Let $M \xrightarrow{g} N$ be a surjective irreducible map, where $M$ and $N$ are indecomposable $A$-modules. Suppose that the stable AR-component containing $N$ is of tree class $A_{\infty}^{\infty}$ or $D_{\infty}$, and that $N$ has at least two predecessors. Then the module $\Ker g$ is periodic, i.e.\ $\Omega_A^n( \Ker g) \simeq \Ker g$ for some $n \ge 1$. Moreover, for any $i \in \mathbb{Z}$, the module $\Omega_A^n( \Ker g)$ has simple top, and $a$ divides $\dim \Omega_A^n( \Ker g)$.
\end{lemma} 

\begin{proof}
Denote the kernel of $g$ by $L$. By \cite[Proposition 2.6]{AuslanderReiten}, the module $L$ is indecomposable, and it cannot be projective since the the map $g$ does not split. By Dade's Lemma (cf.\ \cite[Theorem 2.6]{BensonErdmannHolloway}, the rank variety $\V^r_A(L)$ is nonzero, hence there exists a nonzero $c$-tuple $\lambda \in \V^r_A(L)$. As seen from \cite[Proposition 3.1]{BerghErdmann2} and its proof, the module $Au_{\lambda}$ is periodic of period $2$, with $\Omega_A^1(Au_{\lambda}) = Au_{\lambda}^{a-1}$. Moreover, the vector spaces $\uHom_A(Au_{\lambda},L)$ and $\uHom_A(Au_{\lambda}^{a-1},L)$ are both nonzero. The AR-formula then implies that
$$\Ext_A^1( L, \tau (Au_{\lambda}) ) \neq 0 \neq \Ext_A^1(L, \tau (Au_{\lambda}^{a-1})).$$ 

Consider the indecomposable $A$-modules $\tau (Au_{\lambda})$ and $\tau (Au_{\lambda}^{a-1})$. Since $\tau$ and $\Omega_A$ commute in $\underline{\mod} A$, these modules are both $2$-periodic and syzygies of each other. Now $\tau \simeq \Omega_A^2 \nu$, where $\nu$ is the Nakayama automorphism on $A$, and $\Omega_A$ commutes with $\nu$. Therefore $\tau (Au_{\lambda}) \simeq {_{\nu}(Au_{\lambda})}$ and $\tau (Au_{\lambda}^{a-1}) \simeq {_{\nu}(Au_{\lambda}^{a-1})}$. Now let $W_0$ be either of these two modules. Then $\tau^n W_0 \simeq {_{\nu^n}W_0}$ and $\dim \tau^nW_0 < \dim A$ for any $n \in \mathbb{Z}$. Since $q$ is a root of unity, it follows from \cite[Lemma 3.1]{Bergh} that $\nu$ has finite order, hence every periodic $A$-module is also $\tau$-periodic. Therefore, there exists an integer $n \ge 1$ such that $\tau^n W_0 \simeq W_0$. 

Denote the module $W_0 \oplus \tau W_0 \oplus \cdots \oplus \tau^{n-1}W_0$ by $W$, and note that $\tau W \simeq W$. Furthermore, suppose that $\dim L \ge \dim A$.
Then since $\dim \tau^i W_0 < \dim A \le \dim L$ for any $i$, there cannot exist a monomorphism $L \to \tau^i W_0$. Therefore, the exact sequence 
$$0 \to L \to M \xrightarrow{g} N \to 0$$
and our module $W$ satisfy all the assumptions of Lemma \ref{Extisomorphism}. Moreover, there cannot exist a monomorphism $L \to \Omega_A^{-1}(W_0)$, and so $\Ext_A^1(L,W_0)=0$ by Lemma \ref{Extvanishing}. This is not the case, hence $\dim L < \dim A$.

For any integer $n \in \mathbb{Z}$, the functor $\Omega_A^n$ is an equivalence on $\underline{\mod}A$. Therefore $\Omega_A^n(L)$ is the kernel of a surjective irreducible map in $\Omega_A^n( \Theta )$, where $\Theta$ is the stable AR-component containing $N$. Moreover, the tree class of $\Omega_A^n( \Theta )$ is the same as that of $\Theta$. The above proof then shows that $\dim \Omega_A^n(L) < \dim A$, and so it follows from \cite[Theorem 5.3]{EHSST} that $L$ is periodic (recall that our algebra $A$ satisfies the finite generation hypotheses on Hochschild cohomology).  

Suppose now that for some $n$, the top of $\Omega_A^n(L)$ is not simple. Let
$$0 \to \Omega_A^{n+1}(L) \to A^{\beta} \to \Omega_A^n(L) \to 0$$
be the projective cover of $\Omega_A^n(L)$ (our algebra $A$ is local, so every projective module is free). Since the top of $\Omega_A^n(L)$ is not simple, the integer $\beta$ is at least two, contradicting the fact that $\dim \Omega_A^n(L) < \dim A$ and $\dim \Omega_A^{n+1}(L) < \dim A$. Consequently, the top of $\Omega_A^n(L)$ must be simple.

Finally, consider the rank variety $\V_A^r \left ( \Omega_A^n(L) \right )$ for any $n$. By \cite[Corollary 3.7]{BerghErdmann2}, this variety is one-dimensional, and is therefore not the whole affine space $k^c$. Choose therefore a nonzero $c$-tuple $\mu \in k^c \setminus \V_A^r \left ( \Omega_A^n(L) \right )$. By definition, the $k[ u_{\mu} ]$-module  $\Omega_A^n(L)$ is free, hence the dimension of $k[ u_{\mu} ]$, namely $a$,  divides that of $\Omega_A^n(L)$. 
\end{proof}

Before proving our next result, we recall some basic facts on the truncated polynomial ring $\Gamma = k[x]/(x^a)$. This selfinjective algebra has $a$ isomorphism classes of indecomposable modules, namely $M_i = k[x]/(x^i)$ for $1 \le i \le a$. Each of these is $2$-periodic, with $\Omega_{\Gamma}^1 (M_i) = M_{a-i}$ for $1 \le i \le a-1$. Moreover, when $i<j$, then the stable morphism space $\uHom_{\Gamma}(M_i, M_j)$ is nonzero.

We are now ready to prove the first of our major results. It shows that when our algebra $A$ is not the exterior algebra on $c$ generators, then the tree class of every non-$\tau$-periodic stable AR-component is $A_{\infty}$.

\begin{theorem}\label{exponent}
Let $\Theta$ be a non-$\tau$-periodic component of the stable AR-quiver of $A$. If $a \ge 3$, then $\Theta$ is of tree class $A_{\infty}$.
\end{theorem}

\begin{proof}
Suppose the tree class of $\Theta$ is not $A_{\infty}$. Then by Lemma \ref{noteuclidean} and \cite[Main Theorem]{KernerZacharia}, the tree class is either $A_{\infty}^{\infty}$ or $D_{\infty}$. Choose an indecomposable module $N \in \Theta$ with at least two predecessors, together with a surjective irreducible map $M \xrightarrow{g} N$, where $M$ is indecomposable. Denote the kernel of $g$ by $L$; by Lemma \ref{periodic}, this module is periodic, has simple top, and its dimension is divisible by $a$.

Since the top of $L$ is simple, the equality $\dim L + \dim \Omega_A^1(L) = \dim A$ holds. We may therefore assume that $\dim L \ge \frac{1}{2} \dim A$; if not, then we consider the module $\Omega_A^1(L)$ in the component $\Omega_A^1( \Theta )$ instead. Now choose a nonzero $c$-tuple $\lambda \in \V_A^r(L)$, and for each $1 \le i \le a$, let $M_i$ be the indecomposable $k[u_{\lambda}]$-module of dimension $i$. The $k[u_{\lambda}]$-module $M_1$ is the image of the map $k[u_{\lambda}] \xrightarrow{\cdot u_{\lambda}^{a-1}} k[u_{\lambda}]$, and tensoring this map with the free $k[u_{\lambda}]$-module $A$ we obtain a commutative diagram
$$\xymatrix@C=40pt{
A \otimes_{k[u_{\lambda}]} k[u_{\lambda}] \ar[r]^{1 \otimes \cdot u_{\lambda}^{a-1}} \ar[d]^{\wr} & A \otimes_{k[u_{\lambda}]} k[u_{\lambda}] \ar[d]^{\wr} \\
A \ar[r]^{\cdot u_{\lambda}^{a-1}} & A }$$
in which the vertical arrows are isomorphisms. Thus the $A$-module $W_0= Au_{\lambda}^{a-1}$ is induced from $M_1$, i.e.\ $W_0 \simeq A \otimes_{k[u_{\lambda}]} M_1$,
and its dimension is $a^{c-1}$. In particular, the inequality $\dim W_0 \le \frac{1}{2} \dim A$ holds, since $a \ge 2$. 

As a module over $k[u_{\lambda}]$, our module $L$ has a decomposition $L \simeq \oplus_{i=1}^a M_i^{n_i}$. Since $W_0$ and $M_1$ are $2$-periodic over $A$ and $k[u_{\lambda}]$, respectively, the Eckmann-Shapiro Lemma (cf.\ \cite[Corollary 2.8.4]{Benson2}) gives
\begin{eqnarray*}
\uHom_A( W_0,L) & \simeq & \uHom_A( \Omega_A^2(W_0),L) \\
& \simeq & \Ext_A^2( W_0,L) \\
& \simeq & \Ext_{k[u_{\lambda}]}^2( M_{a-1}, L) \\
& \simeq & \uHom_{k[u_{\lambda}]}( \Omega_{k[u_{\lambda}]}^2(M_{a-1}),L) \\
& \simeq & \uHom_{k[u_{\lambda}]}(M_{a-1},L) \\
& \simeq & \bigoplus_{i=1}^{a-1} \uHom_{k[u_{\lambda}]}(M_{a-1}, M_i^{n_i}) \\
& \simeq & \bigoplus_{i=1}^{a-1} \uHom_{k[u_{\lambda}]} ( \Omega_{k[u_{\lambda}]}^1(M_{a-1}), \Omega_{k[u_{\lambda}]}^1(M_i^{n_i})) \\
& \simeq & \bigoplus_{i=1}^{a-1} \uHom_{k[u_{\lambda}]} (M_1, M_{a-1-i}^{n_i}).
\end{eqnarray*}
From the AR-formula and the discussion prior to this theorem we now see that
$\dim \Ext_A^1 (L, \tau W_0 ) \ge \sum_{i=1}^{a-1}n_i$. Recall from the proof of Lemma \ref{periodic} that $\tau^i W_0 \simeq {_{\nu^i}(W_0)}$, where $\nu$ is the Nakayama automorphism of $A$. Moreover, this automorphism has finite order, say $n$, and so the module $W = W_0 \oplus {_{\nu}(W_0)} \oplus \cdots \oplus {_{\nu^{n-1}}(W_0)}$ satisfies $\tau W \simeq W$. 

We now claim that there must exist a monomorphism $L \to {_{\nu}(W_0)}$.  If not, then all the assumptions of Lemma \ref{Extisomorphism} are satisfied, and so by Lemma \ref{Extvanishing} every nonzero element in $\Ext_A^1(L, {_{\nu}(W_0)})$ is represented by a monomorphism $L \to \Omega_A^{-1}( {_{\nu}(W_0)} )$. Given any two such monomorphisms $\phi_1$ and $\phi_2$, there exists some scalar $\alpha \in k$ such that $\phi_1 + \alpha \phi_2$ vanishes on the socle of $L$, since the latter is simple by Lemma \ref{periodic}. Therefore any set of two elements in $\Ext_A^1(L, {_{\nu}(W_0)})$ is linearly independent, and so the vector space $\Ext_A^1(L, {_{\nu}(W_0)})$ is at most one dimensional. By \cite[Proposition 3.1]{BerghErdmann2}, this vector space is nonzero, hence its dimension is exactly one. However, we saw above that $\dim \Ext_A^1(L, {_{\nu}(W_0)}) \ge \sum_{i=1}^{a-1}n_i$. Since by definition $L$ is not a projective $k[u_{\lambda}]$-module, the sum is nonzero, and so $\sum_{i=1}^{a-1}n_i =1$. This shows that, as a $k[u_{\lambda}]$-module, the module $L$ is the direct sum of exactly one of $M_1, \dots, M_{a-1}$ and possibly some copies of $M_a$. But this is impossible since $a$ divides $\dim L$ by Lemma \ref{periodic}. Consequently, there must exist a monomorphism $L \to {_{\nu}(W_0)}$. This gives inequalities
$$\dim W_0 \le \frac{1}{2} \dim A \le \dim L \le \dim W_0,$$
and so since $\dim W_0 = a^{c-1}$ we see that $a=2$.
\end{proof}

Our next aim is to show that the conclusion of Theorem \ref{exponent} holds when our algebra $A$ has at least three generators, i.e.\ when $c \ge 3$. In order to do this, we first prove a result on the dimension of certain extension groups. 

\begin{lemma}\label{extdimension}
Suppose that $a=2$, i.e.\ that $A$ is the exterior algebra on $c$ variables. Let $M \xrightarrow{g} N$ be a surjective  irreducible map, with $M$ and $N$ indecomposable $A$-modules. Furthermore, suppose the stable AR-component containing $N$ is of tree class $A_{\infty}^{\infty}$ or $D_{\infty}$, and that $N$ has at least two predecessors. Finally, let $L$ be the kernel of $g$, and define
$$\tilde{L} = \left \{
\begin{array}{ll}
L & \text{ if $c$ is odd} \\
L \oplus \tau L & \text{ if $c$ is even.}
\end{array} 
\right.$$
Then $\dim \Ext_A^1(L, \tilde{L})=2$.
\end{lemma}

\begin{proof}
Let $\lambda$ be a nonzero element in $\V_A^r(L)$. From the proof of Theorem  \ref{exponent}, we know that $L \simeq {_{\nu}(Au_{\lambda})}$, where $\nu$ is the Nakayama automorphism of $A$. By \cite[Lemma 3.1]{Bergh}, this automorphism is the identity when $c$ is odd, whereas $\nu^2 = 1$ when $c$ is even. Hence $\tilde{L} \simeq \tau \tilde{L}$ in either case. Note also that $\Omega_A^1(L) \simeq L$, since we are assuming that $a=2$. 

Applying $\Hom_A(-, \tilde{L})$ to the exact sequence
\begin{equation}\label{ES}
0 \to L \xrightarrow{f} M \xrightarrow{g} N \to 0 \tag{$\dagger$}
\end{equation}
gives a long exact sequence
$$0 \to (N, \tilde{L} ) \to (M, \tilde{L} ) \xrightarrow{(f, \tilde{L} )} (L, \tilde{L} ) \xrightarrow{\delta} {^1(N, \tilde{L} )} \to {^1(M, \tilde{L} )} \xrightarrow{{^1(f, \tilde{L} )}} {^1(L, \tilde{L} )} \to \cdots$$
in cohomology, where $(-,-) = \Hom_A(-,-)$ and ${^1(-,-)} = \Ext_A^1(-,-)$. We show that the image of the map $\Ext_A^1(f, \tilde{L})$ is one dimensional. Consider first the image of the map $\Hom_A(f, \tilde{L})$. If $h \in \End_A(L)$ is invertible, then it cannot be in the image of $\Hom_A(f, \tilde{L})$ since (\ref{ES}) does not split. On the other hand, if $h$ is not invertible, then it is not a monomorphism, and hence by \cite[Proposition 1.1]{Erdmann2} it does belong to the image of $\Hom_A(f, \tilde{L})$. If $c$ is even, then $L$ is not isomorphic to $\tau L$, since $\tau L = {_{\nu}L}$ and $\nu$ has order two. Then since the dimension of $L$ equals that of $\tau L$, there cannot exist any monomorphism $L \to \tau L$. Therefore, as above, $\Hom_A(L, \tau L)$ is in the image of $\Hom_A(f, \tilde{L})$. Since our field $k$ is algebraically closed, the dimension of $\End_A(L) / \rad_A(L,L)$ is one, hence 
$$\dim \Im \Hom_A(f, \tilde{L}) = \dim \Hom_A(L, \tilde{L} ) -1$$
regardless of the parity of $c$. Now we can easily compute the dimension of the image of $\Ext_A^1(f, \tilde{L})$; by Lemma \ref{tauconstant} and the AR-formula, the dimension of $\Ext_A^1(M, \tilde{L})$ equals that of $\Ext_A^1(N, \tilde{L})$, and so $\dim \Im \Ext_A^1(f, \tilde{L}) =1$.

Since $\Omega_A^1( \tilde{L} ) \simeq \tilde{L}$, there is a commutative diagram
$$\xymatrix{
{^1(M, \tilde{L} )} \ar[r]^{{^1(f,\tilde{L})}} \ar@{=}[d] & {^1(L, \tilde{L} )} \ar@{=}[d] \ar[r] & {^2(N, \tilde{L} )} \ar[r] \ar[d]^{\wr} & {^2(M, \tilde{L} )} \ar[r]^{{^2(f,\tilde{L})}} \ar[d]^{\wr} & {^2(L, \tilde{L} )} \ar[d]^{\wr} \\
{^1(M, \tilde{L} )} \ar[r]^{{^1(f,\tilde{L})}} & {^1(L, \tilde{L} )} \ar[r] & {^1(N, \tilde{L} )} \ar[r] & {^1(M, \tilde{L} )} \ar[r]^{{^1(f,\tilde{L})}} & {^1(L, \tilde{L} )} }$$
with exact rows, in which all the vertical maps are isomorphisms. From what we showed above, the dimensions of $\Ext_A^2(M, \tilde{L})$ and $\Ext_A^2(N, \tilde{L})$ must be equal, and $\dim \Im \Ext_A^2(f, \tilde{L}) =1$. By counting dimensions in the top exact sequence, we see that $\dim \Ext_A^1(L, \tilde{L})=2$.
\end{proof}

We are now ready to prove the second of our major results. It shows that the conclusion in Theorem \ref{exponent} holds also when our algebra $A$ has at least three generators.

\begin{theorem}\label{codimension}
Let $\Theta$ be a non-$\tau$-periodic component of the stable AR-quiver of $A$. If $c \ge 3$, then $\Theta$ is of tree class $A_{\infty}$.
\end{theorem}

\begin{proof}
If $a \ge 3$, then the result holds by Theorem \ref{exponent}, hence we may assume that $a=2$. We proceed as in the beginning of the proof of Theorem \ref{exponent}, by assuming that the tree class of $\Theta$ is not $A_{\infty}$. Then by Lemma \ref{noteuclidean} and \cite[Main Theorem]{KernerZacharia}, the tree class is either $A_{\infty}^{\infty}$ or $D_{\infty}$. Choose an indecomposable module $N \in \Theta$ with at least two predecessors, together with a surjective irreducible map $M \xrightarrow{g} N$, where $M$ is indecomposable. Denote the kernel of $g$ by $L$, and define $\tilde{L}$ as in Lemma \ref{extdimension}. We shall use the fact that $\dim \Ext_A^1(L, \tilde{L})=2$ to obtain that $c$ must equal $2$.

Let $\lambda = ( \lambda_1, \dots, \lambda_c )$ be a nonzero element in $\V_A^r(L)$. As in the proofs of Theorem \ref{exponent} and Lemma \ref{extdimension}, we know that $L$ is isomorphic to ${_{\nu}(Au_{\lambda})}$, where $\nu$ is the Nakayama automorphism. Since $\Ext_A^i({_{\phi}X},{_{\phi}Y}) \simeq \Ext_A^i(X,Y)$ for any modules $X,Y$ and any automorphism $A \xrightarrow{\phi} A$, we see that
$$2 = \dim \Ext_A^1(L, \tilde{L}) \ge \dim \Ext_A^1(L,L) = \dim \Ext_A^1(Au_{\lambda},Au_{\lambda}).$$
Now recall from the proof of Theorem \ref{exponent} that the $A$-module $Au_{\lambda}$ is induced from the one-dimensional $k[u_{\lambda}]$-module $M_1$, i.e.\ $Au_{\lambda} \simeq A \otimes_{k[u_{\lambda}]} M_1$. The Eckmann-Shapiro Lemma therefore gives
$$\Ext_A^1(Au_{\lambda},Au_{\lambda}) \simeq \Ext_{k[u_{\lambda}]}^1(M_1,Au_{\lambda}).$$
What is the structure of $Au_{\lambda}$ as a module over $k[u_{\lambda}]$? Since $a=2$, our algebra $A$ is the exterior algebra on $c$ generators, and so
$$u_{\lambda}x_i = ( \lambda_1 x_1 + \cdots + \lambda_c x_c) x_i = -x_i ( \lambda_1 x_1 + \cdots + \lambda_c x_c) =-x_iu_{\lambda}$$
for every $1 \le i \le c$. Since $u_{\lambda}^2=0$ by \cite[Lemma 2.3]{BensonErdmannHolloway}, this implies that $u_{\lambda}(Au_{\lambda}) = Au_{\lambda}^2 =0$. Consequently, the $k[u_{\lambda}]$-module $Au_{\lambda}$ must consist entirely of copies of $M_1$, and the number of such copies must be $\dim Au_{\lambda}$. We know that, since $a=2$, the dimension of $Au_{\lambda}$ is $2^{c-1}$, and so since the $k[u_{\lambda}]$-module $M_1$ is $1$-periodic, we obtain
$$2 \ge \dim \Ext_{k[u_{\lambda}]}^1(M_1,Au_{\lambda}) = \dim \uHom_{k[u_{\lambda}]}(M_1, \bigoplus_{2^{c-1}}M_1) = 2^{c-1}.$$
This shows that $c=2$.
\end{proof}

Having proved Theorem \ref{exponent} and Theorem \ref{codimension}, we can now prove the main result. Here we deal with quantum complete intersections whose defining commutators are arbitrary roots of unity, i.e.\ we do not require that they are equal. We prove that when the algebra is wild, then every stable AR-component is of tree class $A_{\infty}$, whereas in the tame case, there is one non-$\tau$-periodic component, and its tree class is $\tilde{A}_{12}$. 

\begin{theorem}\label{maintheorem}
Let $k$ be an algebraically closed field, and $c$ and $a$ two integers, both at least $2$. Let $(q_{ij})$ be a $c \times c$ commutation matrix over $k$, with all the $q_{ij}$ roots of unity, and define the quantum complete intersection $\La$ by 
$$\La = k \langle x_1, \dots, x_c \rangle / (x_i^a, x_ix_j-q_{ij}x_jx_i).$$
\begin{enumerate}
\item If either $a \ge 3$ or $c \ge 3$, then every stable AR-component of $\La$ is of tree class $A_{\infty}$.
\item If $a=2=c$, then $\La$ has one stable AR-component of tree class is $\tilde{A}_{12}$, whereas all the other are of tree class $A_{\infty}$.
\end{enumerate}
\end{theorem}

\begin{proof}
We start by showing that we can reduce to a \emph{homogeneous} quantum complete intersection having parameters $a$ and $c$. Namely, it is shown in \cite[Section 4]{BensonErdmannHolloway} that $\La$ is Morita equivalent to a skew group algebra $A [G]$, where $A$ is a homogeneous quantum complete intersection with parameters $a$ and $c$, and where the order of the group $G$ is invertible in $k$. Thus by \cite{Solberg}, we may assume that $\La =A$.

If either $a \ge 3$ or $c \ge 3$, then from Theorem \ref{exponent} and Theorem \ref{codimension} we know that every non-$\tau$-periodic stable component is of tree class $A_{\infty}$. If $a=2=c$, then it is well known that $\La$ has one stable component which is not $\tau$-periodic, and that its tree class is $\tilde{A}_{12}$. Finally, all $\tau$-periodic components are of tree class $A_{\infty}$.
\end{proof}

\end{document}